\documentclass[11pt,twoside]{amsart}

\usepackage{xspace}
\usepackage{amssymb}
\usepackage{bbm}
\usepackage{graphicx}
\usepackage{url}
\usepackage{pifont}

\numberwithin{equation}{section}

\newcommand{\mi}{\mathbbm{i}}

\newcommand{\bbC}{\mathbb{C}}

\newcommand{\bbN}{\mathbb{N}}

\newcommand{\bbP}{\mathbb{P}}

\newcommand{\bbR}{\mathbb{R}}
\newcommand{\bbS}{\mathbb{S}}

\newcommand{\bbZ}{\mathbb{Z}}

\theoremstyle{plain}
\newtheorem{theorem}{Theorem}[section]
\newtheorem*{theorem*}{Theorem}

\newtheorem*{corollary*}{Corollary}

\newtheorem*{proposition*}{Proposition}
\newtheorem{lemma}[theorem]{Lemma}
\newtheorem*{lemma*}{Lemma}

\newtheorem*{example*}{Example}

\newtheorem*{definition*}{Definition}

\newtheorem*{notation*}{Notation}

\newtheorem*{remark*}{Remark}

\numberwithin{figure}{section}

\newcommand{\cmc}{{\sc{cmc}}\xspace}

\newlength{\PWIDTH}
\newlength{\PSPACE}
\newcommand{\spaceperiod}{\,\,.}

\title{Bubbletons are not embedded}

\author{Martin Kilian}
\address{M. Kilian, Department of Mathematics,
University College Cork, Ireland.}
\email{m.kilian@ucc.ie}

\begin{document}


\begin{abstract}
We discuss constant mean curvature bubbletons in Euclidean 3-space via dressing with simple factors, and prove that single bubbletons are not embedded.
\end{abstract}


\maketitle

\section*{Introduction}

A key feature of an integrable system is the presence of an algebraic transformation method which generates new solutions from old ones. In particular even by starting with a trivial solution one obtains a hierarchy of interesting global solutions. For the KdV equation one thus obtains the solitons via a B\"acklund transform. Solitons are solitary traveling waves with localized energy that are stable when interacting with each other. Many of the modern techniques in integrable systems theory stem from classical surface theory, developed by B\"acklund, Bianchi and Darboux amongst others for the structure equations of special surface classes.

Away from umbilic points the structure equation of constant mean curvature (\cmc) surfaces is the sinh-Gordon equation, whose trivial solution gives rise to the round cylinder. The term 'bubbleton' is due to Sterling and Wente \cite{SteW:bub}, and the bubbletons are the solitons of the sinh-Gordon equation. The single bubbletons are obtained by transforming the standard cylinder by a Bianchi-B\"acklund transform. The resulting transformed \cmc cylinder globally looks like the standard cylinder except for a localized part in which bubble-like pieces are added to the underlying surface, see figure \ref{fig:bubbleton}. A video of how bubbles interact when they move through each other can be seen at \cite{Sch:bub_movie}.

Recently the classical transformations have received a treatment from the modern point of view of dressing \cite{Bur, BurDLQ, BurP:dre, DorH:dre, DorK, HerP, Kil:thesis, KilSS, Kob:bub, KobI, TerU}. Bubbletons can be realized by dressing the round cylinder by a class of very simple maps, called simple factors \cite{TerU}. By repeatedly applying the Bianchi-B\"{a}cklund transformation to the standard round cylinder one produces the 'multi-bubbletons' classified by Sterling and Wente \cite{SteW:bub}. While graphics of these surfaces clearly suggest that they are not embedded (see figure \ref{fig:bubbleton}), there does not seem to be a direct proof of this fact in the literature. An indirect proof when the target is the 3-sphere is given in \cite{KilS}.

The purpose here is to prove that a single bubbleton is not embedded. This is proven by showing that single bubbletons possess a planar curve which is not homologous to zero on the surface. It is shown that this curve has turning number at least three, which implies that the surface cannot be embedded. Furthermore the choice of the 'singularity' in the simple factor is reflected in the geometry of the resulting bubbleton. The monotone sequence of the singularities are indexed by an integer $K \in \bbN$ for $K \geq 2$, and $K$ is the number of 'bubbles' of the bubbleton, and $2K-1$ turns out to be the turning number of the planar curve.

\section{The round cylinder}

If $\mi = \sqrt{-1}$, then in the spinor representation of Euclidean 3-space $\bbR^3$ we identify $\bbR^3 \cong \mathfrak{su}_2$ via
\begin{equation*}
    (x_1,\, x_2,\,x_3 ) \cong \begin{pmatrix} \mi x_3 & x_1 + \mi x_2  \\ -x_1 + \mi x_2 & -\mi x_3 \end{pmatrix} \,.
\end{equation*}
The extended frame of a round cylinder (up to isometry and conformal change of coordinate) is
\begin{equation} \label{eq:flat-frame}
    F_\lambda(z)
        = \begin{pmatrix} \cos \mu_\lambda  & \mi\,\lambda^{-1/2}\sin \mu_\lambda \\ \mi\,\lambda^{1/2} \sin \mu_\lambda & \cos \mu_\lambda \end{pmatrix} \,,
\end{equation}
where
\begin{equation} \label{eq:mu-function}
        \mu_\lambda = \mu_\lambda(z) = \tfrac{\pi}{2} \bigl( \,\, z\,\lambda^{-1/2} + \bar{z}\,\lambda^{1/2}) \spaceperiod
\end{equation}
The Sym-Bobenko formula \cite{Sym, Bob:cmc} for a \cmc surface in Euclidean 3-space $\bbR^3$ is a formula of the immersion in terms of its extended frame $F_\lambda$. In our conventions \cite{SKKR}, the associated family with constant mean curvature $H \in \bbR^\times$ is given by
\begin{equation} \label{eq:sym-bobenko}
    f_\lambda(z) = -2 \mi \lambda\,H^{-1} F^{\,\prime}_\lambda(z) F_\lambda^{-1}(z)
\end{equation}
where $F^{\,\prime}_\lambda$ denotes the derivative with respect to $\lambda$. If we pick one member of the associated family $f_\lambda(z),\,\lambda \in \bbS^1$ and choose $\lambda =1$, and insert the extended frame \eqref{eq:flat-frame} of the standard round cylinder, we obtain
\begin{equation*}
    f_1(x,\,y) = H^{-1} \begin{pmatrix}
    \mi \sin ^2(\pi x) & -\cos(\pi x)\sin(\pi x) - \pi\mi y\\
     \cos(\pi x)\sin(\pi x) - \pi\mi y & -\mi \sin ^2(\pi x)
    \end{pmatrix}
     \cong \frac{1}{2H} \begin{pmatrix} \sin (2\pi x) \\ -2 \pi y \\1 - \cos(2\pi x) \end{pmatrix}\,.
\end{equation*}
This is clearly a round cylinder which is generated by a circle in the $x_1 x_3$-plane of radius $1/(2|H|)$ centered at the point $(0,\,1/(2H))$ parallel translated along the $x_2$-axis. Any curve $x_2 = c$ for some constant $c \in \bbR$ is not contractable on the cylinder. By restricting the $x_1$-coordinate to any interval of length one we obtain an embedded round cylinder.

In general, or when the parametrization is less explicit, one can describe the period problem as follows. Suppose we have an extended frame $F_\lambda$ and an associated family of \cmc surfaces as in \eqref{eq:sym-bobenko}. Then periodicity $f_\lambda (z + \tau) = f(z)$ for all $z \in \bbC$ can in general not hold for all $\lambda \in \bbS^1$. But if we fix $\lambda_0 \in \bbS^1$, then periodicity reads
\begin{equation} \label{eq:periodicity}
    F^{\,\prime}_{\lambda_0}(z + \tau) F_{\lambda_0}^{-1}(z + \tau) = F^{\,\prime}_{\lambda_0}(z) F_{\lambda_0}^{-1}(z)\,.
\end{equation}
If we define the monodromy matrix $M_\lambda(\tau)$ with respect to the translation $z \mapsto z +\tau$ of $F_\lambda$ by
\begin{equation} \label{eq:monodromy}
    M_\lambda (\tau) = F_\lambda (z +\tau) F_\lambda^{-1}(z)
\end{equation}
then the period problem \eqref{eq:periodicity} reads
\begin{equation} \label{eq:closing0}
    M_{\lambda_0} (\tau) = \pm \mathbbm{1} \mbox{ and }  M^{\,\prime}_{\lambda_0} (\tau)  = 0\,.
\end{equation}
The monodromy matrix is not well defined, since it depends on the choice of a base point, and so is only defined up to conjugacy. However, the periodicity conditions \eqref{eq:closing0} are invariant under conjugation.

Since we need some of the above in the special case of the round cylinder, let us specialize again in this case. If $F_\lambda$ is the extended frame \eqref{eq:flat-frame} of a round cylinder, we choose the base point $z_0 =0$ and note that $F_\lambda(0) = \mathbbm{1}$ for all $\lambda \in \bbC^\times$. As before pick $\lambda_0 =1$. The monodromy of $F_\lambda(z)$ with respect to the translation $\tau:z \mapsto z + 1$ is then
\begin{equation} \label{eq:monodromy-flat}
    M_\lambda(\tau) = F_\lambda(1)\,,
\end{equation}
and a quick computation confirms that $F_1(1) = - \mathbbm{1}$ and $F^{\,\prime}_1(1) = 0$.
%
%
\begin{figure}[t]
  \centering
  \includegraphics[scale=0.35]{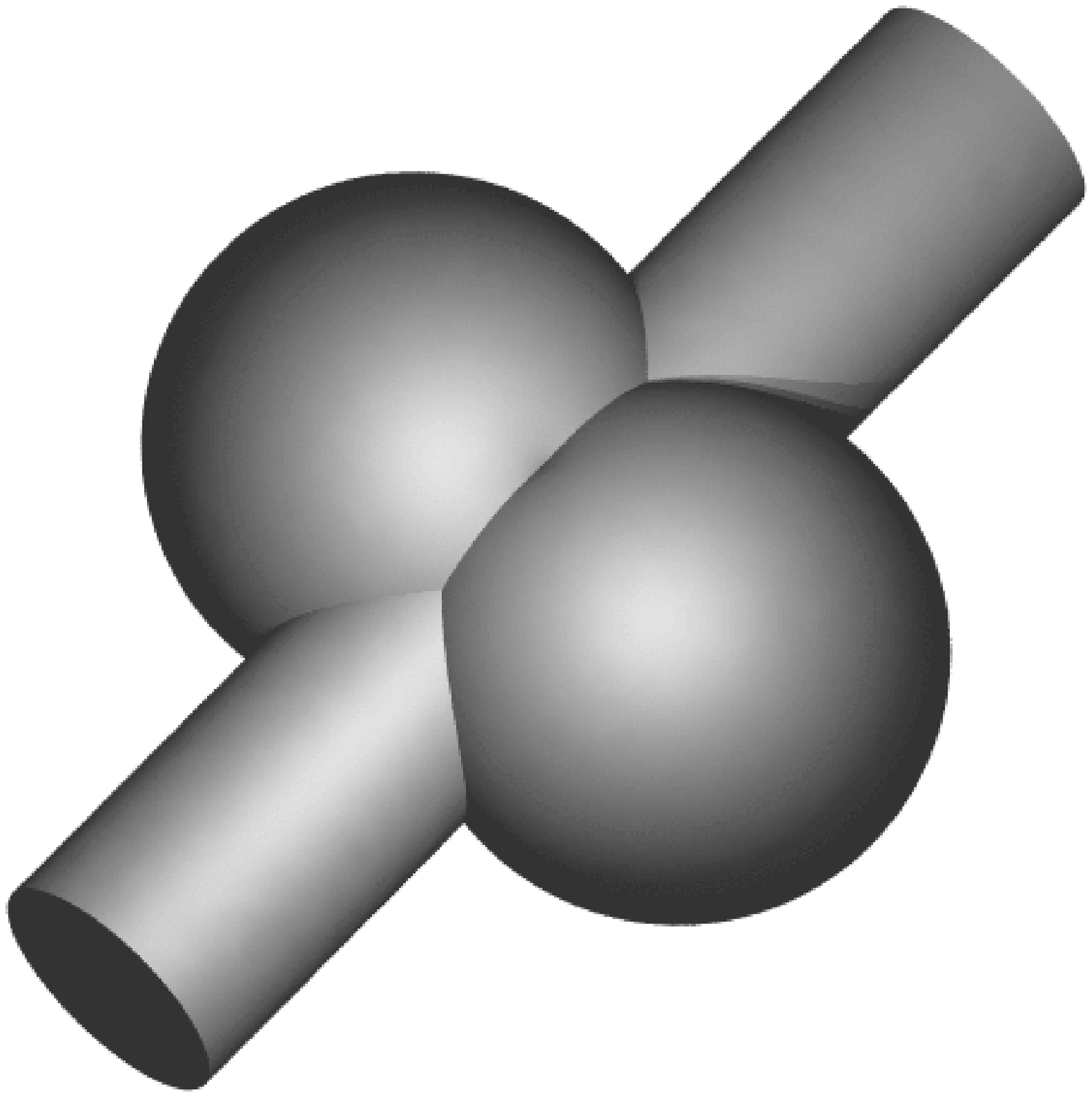}
  \includegraphics[scale=0.35]{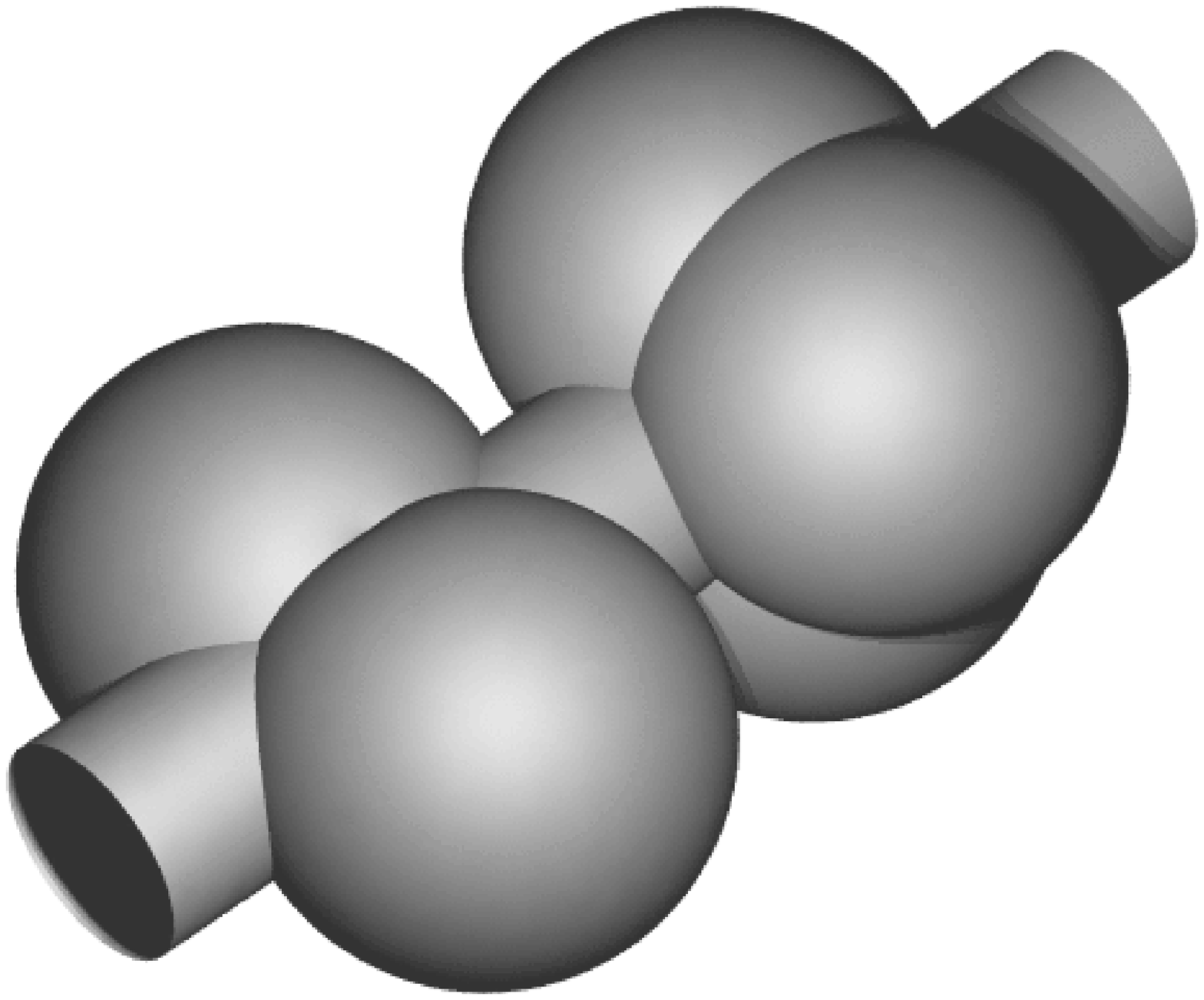}
  \caption{Parts of a two-lobed bubbleton, and a multi-bubbleton with 2 and 3 lobes. Further graphics of bubbletons can be viewed at \cite{Sch:web, Sch:gallery}.
    \label{fig:bubbleton}}
\end{figure}
%

%

\section{Simple factors}

There is a deformation technique in the theory of harmonic maps called \emph{dressing} \cite{BurP:dre}. In particular, dressing by specific very simple maps corresponds to the classical Bianchi-B\"acklund transformation \cite{KobI}, and amounts to adding 'bubbletons' to the standard round cylinder - these simple maps are called \emph{simple factors} \cite{TerU}. Let us briefly review the theory of simple factors in the context of \cmc surfaces in $\bbR^3$. Let $\pi_L:\bbC^2 \to L$ be the hermitian projection onto a line $L \in \bbC \bbP^1$, and $\pi_L^\perp = \mathbbm{1}- \pi_L$. For $\alpha \in \bbC^\times$, set
\begin{equation} \label{eq:simple}
  \psi_{L,\,\alpha} (\lambda) =
  \pi_L +  \frac{\alpha - \lambda}
  {1-\bar{\alpha}\,\lambda}\,\pi_L^\perp.
\end{equation}
To normalize make the determinant equal to 1 and do a Gram-Schmidt factorization at $\lambda = 0$ to obtain
$(\det \psi_{L,\,\alpha}(0))^{-1/2}\,\psi_{L,\,\alpha}(0) = Q\,R$ with $Q \in \mathrm{SU}_2$ and $R \in \mathrm{SL}_2$ upper triangular with positive real entries on the diagonal. A \emph{simple factor} is a map of the form
  \begin{equation} \label{eq:simple-factor}
    h_{L,\,\alpha} =
    (\det \psi_{L,\,\alpha})^{-1/2}\,
    Q^{-1} \psi_{L,\,\alpha} \spaceperiod
  \end{equation}
By Proposition 4.2 in \cite{TerU} dressing by simple factors is explicit, and adapted to the case at hand in Theorem 1.2 in \cite{KilSS}:
Generally, suppose that $F_\lambda$ is an extended frame, and $h_{L,\,\alpha}$ a simple factor with $\alpha \in \bbC^\times,\,|\alpha|<1$, and $L \in \bbP^1$. Then the dressed extended frame is given by dressing on an $r$-circle with $r<|\alpha|<1$, and is
  \begin{equation}\label{eq:dressing}
    h_{L,\alpha}\# \,F_\lambda =
    h_{L,\alpha} \, F_\lambda \, h^{-1}_{\widetilde{L},\alpha}
    \mbox{ with } \widetilde{L} = \overline{F_\alpha(z)}^tL.
  \end{equation}
We next show that to obtain the single bubbletons we can choose diagonal simple factors with very specific singularities $\alpha$. 
\begin{lemma}
Up to isometry and conformal coordinate change any single bubbleton can be obtained by dressing the round cylinder by a simple factor $h_{L,\,\alpha}$ with line $L=[1 : 0]$, so of the form
\begin{equation} \label{eq:simple-matrix}
     h_{L,\,\alpha} = \left(\tfrac{\alpha - \lambda}{1- \alpha\,\lambda}\right)^{-1/2} \begin{pmatrix} 1 & 0 \\ 0 & \frac{\alpha - \lambda}{1- \alpha\,\lambda} \end{pmatrix}\,,
\end{equation}
with real $\alpha \in (0,\,1)$ of the form 
\begin{equation} \label{eq:alpha}
    \alpha = 2K^2 -1 - 2K \sqrt{K^2-1} \mbox{ for some integer } K\geq 2 \,.
\end{equation}
\end{lemma}
\begin{proof}
If we want to dress the extended frame of the round cylinder, then we will have to choose the line $L$ and the singularity $\alpha$ in such a way that the resulting bubbleton remains periodic with the same period as the underlying round cylinder. To get the conditions on $L$ and $\alpha$ we will need to look at the monodromy of the dressed extended frame \eqref{eq:dressing}, with respect to the same translation $\tau :z \mapsto z +1$. Since $F_\lambda(0) = \mathbbm{1}$ for all $\lambda \in \bbC^\times$, for the dressed frame we also have $h_{L,\alpha}\# \,F_\lambda |_{z=0} = \mathbbm{1}$ since $\widetilde{L} = L$ for $z=0$. Hence with the monodromy $M_\lambda(\tau)$ in \eqref{eq:monodromy-flat} of the round cylinder we obtain for the bubbleton monodromy
\begin{equation} \label{eq:eigenline}
  h_{L,\alpha}\# \,F_\lambda |_{z=1} = h_{L,\alpha} \,M_\lambda(\tau) \, h^{-1}_{L,\alpha} \quad \Longleftrightarrow \quad \overline{M_\alpha(\tau)}^t L = L\,.
\end{equation}
Thus the condition on the line $L$ is that it has to be an eigenline of $\overline{M_\alpha(\tau)}^t$. Now $\mathrm{SU}_2$ acts transitively on $\bbC\bbP^1$ and $ h_{UL,\alpha} = U h_{L,\alpha}U^{-1}$ for any $U \in \mathrm{SU}_2$. Since dressing by $h_{L,\alpha}$ and $U h_{L,\alpha}U^{-1}$ give the same surface up to isometry and conformal coordinate change, we may choose without loss of generality the line $L=[1\,:\,0]$. 

The monodromy of an extended frame of a \cmc surface should be holomorphic in $\lambda \in \bbC^\times$ and unitary for $\lambda \in \bbS^1$. Clearly away from $\lambda = \alpha,\,1/\alpha$ we have
$$
    \overline{h_{L,\alpha}(1/\bar{\lambda})}^t = h_{L,\alpha}^{-1}(\lambda)
$$
so $h_{L,\alpha} M_\lambda(\tau) h_{L,\alpha}^{-1}$ is unitary on the unit circle, if we demand that
\begin{equation} \label{eq:alpha-not-unimodular}
    |\alpha| \neq 1 \,.
\end{equation}
Further $h_{L,\alpha} M_\lambda(\tau) h_{L,\alpha}^{-1}$ is holomorphic for all $\lambda \in \bbC^\times$ away from $\lambda =\alpha,\,\alpha^{-1}$. To make these two singularities removable, we impose the condition that $M_\alpha(\tau) = F_\alpha (1) = \pm \mathbbm{1}$, or equivalently that $\mu_\alpha (1) = \pm 1$ for the function $\mu_\lambda$ in \eqref{eq:mu-function}. This is equivalent to there existing an integer $K \in \bbZ$ such that
\begin{equation}\label{eq:alpha+}
    \alpha^{-1/2} + \alpha^{1/2} = 2K \,.
\end{equation}
Rewriting this as a quadratic equation we obtain for each $K\in\bbZ$ two real solutions 
\begin{equation} \label{eq:double-points}
    \alpha_{\pm} = 2K^2 -1 \pm 2K \sqrt{K^2-1}\,.
\end{equation}
First observe that $\alpha_- = \alpha_+^{-1}$. Now $\psi_{L,\alpha^{-1}} = \psi^{-1}_{L,\alpha}$ and 
\begin{equation*}
    h_{L,\alpha^{-1}} = \sqrt{\tfrac{\alpha - \lambda}{1-\alpha\lambda}}\,\sqrt{\tfrac{1-\alpha\lambda}{\alpha - \lambda}} \,\, h_{L,\alpha}^{-1} = \sqrt{\tfrac{\alpha - \lambda}{1-\alpha\lambda}}\,\sqrt{\tfrac{1-\alpha\lambda}{\alpha - \lambda}} \,\, h_{UL,\alpha} \,\,\mbox{ for } \,\,U = \begin{pmatrix} 0 & -1 \\ 1 & 0 \end{pmatrix} \,.
\end{equation*}
But since
$$
    \left.\sqrt{\tfrac{\alpha - \lambda}{1-\alpha\lambda}}\,\sqrt{\tfrac{1-\alpha\lambda}{\alpha - \lambda}}\,\right|_{\lambda =1} = -1 \mbox{ and } \left.\tfrac{\partial}{\partial \lambda}\right|_{\lambda =1}(\sqrt{\tfrac{\alpha - \lambda}{1-\alpha\lambda}}\,\sqrt{\tfrac{1-\alpha\lambda}{\alpha - \lambda}}) = 0\,, 
$$
dressing by $h_{L,\alpha}$ and $h_{L,\alpha^{-1}}$ gives the same bubbleton up to isometry and conformal coordinate change. Thus $\alpha_+$ and $\alpha_-$ give the same bubbleton, so we may omit the subscript, restrict to non-negative integers $K \geq 0$, and set $\alpha = \alpha_-$, so that $\alpha$ is as in \eqref{eq:alpha}.

Then $\alpha^{-1} - \alpha = (\alpha^{-1/2} + \alpha^{1/2})(\alpha^{-1/2} -\alpha^{1/2}) = 2K( \alpha^{-1/2} -\alpha^{1/2}$, and consequently
\begin{equation*}
    \alpha^{-1/2} - \alpha^{1/2} = 2\,\sqrt{K^2-1}\,.
\end{equation*}
Note that when $K=0,\,\pm 1$ then $\alpha = -1,\,1,\,1$ respectively, and it is not hard to see that $\alpha \notin \bbS^1$ when $|K|\geq 2$. Hence the condition \eqref{eq:alpha-not-unimodular} that $|\alpha | \neq 1$ requires that we impose $K \neq 0,\,\pm 1$, and consequently we may restrict to the case $K \geq 2$. To see that the singularities at $\lambda = \alpha,\,1/\alpha$ are now apparent, we use L'Hoptial's rule to obtain
\begin{equation*} \begin{split}
   &\lim_{\lambda \to \alpha} h_{L,\alpha} \, F_\lambda (1) \,h^{-1}_{L,\alpha} = \lim_{\lambda \to \alpha}\begin{pmatrix} \cos\mu_\lambda(1) & \tfrac{1-\alpha\lambda}{\alpha-\lambda}
   \mi\lambda^{-1/2}\sin\mu_\lambda(1) \\ \tfrac{\alpha-\lambda}{1-\alpha\lambda}
   \mi\lambda^{1/2}\sin\mu_\lambda(1) & \cos\mu_\lambda(1) \end{pmatrix} \\ &= \begin{pmatrix} \pm 1 & 0 \\ 0 & \pm 1 \end{pmatrix} + \mi\alpha^{-1/2}(1-\alpha^2)\lim_{\lambda \to \alpha}\frac{\sin\mu_\lambda(1)}{\alpha-\lambda}
   \begin{pmatrix} 0 & 1 \\ 0 & 0 \end{pmatrix} \\
   &= \begin{pmatrix} \pm 1 & 0 \\ 0 & \pm 1 \end{pmatrix} \pm \mi\alpha^{-1/2}4K(K^2-1)
   \begin{pmatrix} 0 & 1 \\ 0 & 0 \end{pmatrix}\,,
   \end{split}
\end{equation*}
with sign depending on the parity of $K$. A similar computation, or using $\overline{F_{1/\bar{\lambda}}}^{\,t} = F_\lambda^{-1}$ shows that $\lambda = 1/\alpha$ is also a removable singularity. 
\end{proof}
To obtain any multi-bubbleton of finite type one can dress by a finite product $\Pi\, h_{L,\alpha_j}$ of simple factors, but we may use each integer $K \geq 2$ only once, since otherwise the singularities are no longer removable in the dressed monodromy. Hence on a multi-bubbleton each lobe number can only appear once.
\section{The main result}

With the preparations of the preceding two sections we can now prove our main result. 
\begin{figure}[t]
\centering
\includegraphics[width=12.6cm]{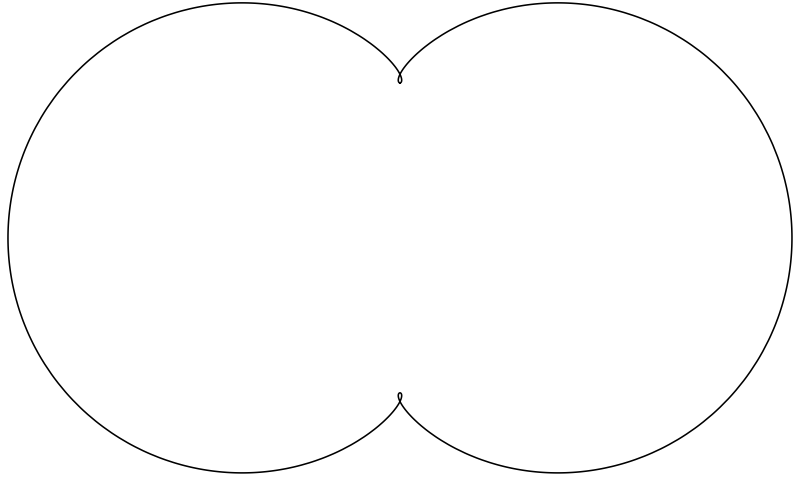}\hspace{0.0cm}
\caption{ \label{fig:curves} Planar curve on the 2-lobed single bubbleton. It has turning number $3$.}
\end{figure}
\begin{theorem}
A single bubbleton is not embedded.
\end{theorem}
\begin{proof}
From \eqref{eq:dressing} the extended frame of a bubbleton is $h_{L,\,\alpha} \, F_\lambda \, h^{-1}_{\widetilde{L},\,\alpha}$ with
\begin{equation}\label{eq:bubbleton-frame}
    \widetilde{L} = \overline{F_\alpha(z)}^{\,\,t} \begin{pmatrix} 1\\0 \end{pmatrix} = \begin{pmatrix} \cos \bar{\mu}_\alpha \\ -\mi \alpha^{-1/2} \sin \bar{\mu}_\alpha \end{pmatrix}\,.
\end{equation}
We next explicitly compute the curve $y=0$ on a bubbleton. We may set the mean curvature to $H=-1/2$. Inserting the bubbleton frame into the Sym-Bobenko formula \eqref{eq:sym-bobenko} gives three terms
\begin{equation} \label{eq:3parts}
    \left. \mi h'_{L,\,\alpha} h_{L,\,\alpha}^{-1} + \mi h_{L,\,\alpha} F^{\,\prime}_\lambda \,F_\lambda^{-1} h_{L,\,\alpha}^{-1} +  \mi h_{L,\,\alpha} F_\lambda h_{\widetilde{L},\,\alpha}^{-1\,\prime} h_{\widetilde{L},\,\alpha} \,F_\lambda^{-1} h_{L,\,\alpha}^{-1}\right|_{\lambda =1,\,y=0}\,.
\end{equation}
In the following we will compute these three terms. First, we have for $h_{L,\alpha}$ in \eqref{eq:simple-matrix} that
\begin{equation} \label{eq:part1}
    \left. \mi\,h^\prime_{L,\alpha} h_{L,\alpha}^{-1} \right|_{\lambda =1} = \frac{K}{2\sqrt{K^2 -1}}
    \begin{pmatrix} -\mi & 0 \\ 0 & \mi \end{pmatrix}\,.
\end{equation}
The second term is
\begin{equation} \label{eq:part2}
    \left. \mi\,h_{L,\alpha} F^{\,\prime}_\lambda  F^{-1} h_{L,\alpha}^{-1}\right|_{\lambda =1,\,y=0} = \frac{1}{2}
    \begin{pmatrix} -\mi\sin^2(\pi x) & -\frac{1}{2}\sin(2\pi x) \\ \frac{1}{2}\sin(2\pi x) & \mi\sin^2(\pi x) \end{pmatrix}\,.
\end{equation}
If $L=[a:b]$, then with respect to the standard basis of $\bbC^2$ the projection $\pi_L$ is
\begin{equation*}
    \pi_L = \frac{1}{|a|^2 + |b|^2}\begin{pmatrix} |a|^2 & a\bar{b} \\ \bar{a}b & |b|^2 \end{pmatrix} \spaceperiod
\end{equation*}
Hence for the line $\widetilde{L}$ in \eqref{eq:bubbleton-frame} we obtain
\begin{equation*}
    \pi_{\widetilde{L}} = \begin{pmatrix}
    \frac{|\cos\mu_\alpha|^2}{|\cos\mu_\alpha|^2 + \alpha^{-1} |\sin\mu_\alpha |^2} & \frac{\mi \alpha^{-1/2}\sin\mu_\alpha \cos\bar{\mu}_\alpha}{|\cos\mu_\alpha|^2 + \alpha^{-1} |\sin\mu_\alpha |^2} \\ \frac{-\mi\alpha^{-1/2}\sin\bar{\mu}_\alpha \cos\mu_\alpha}{|\cos\mu_\alpha|^2 + \alpha^{-1} |\sin\mu_\alpha |^2} & \frac{\alpha^{-1} |\sin\mu_\alpha|^2}{|\cos\mu_\alpha|^2 + \alpha^{-1} |\sin\mu_\alpha |^2} \end{pmatrix} \spaceperiod
\end{equation*}
The function $\mu_\lambda$ defined in \eqref{eq:mu-function} evaluated at $\lambda = \alpha$ along $y=0$ reads $\mu_\alpha = \pi K\,x$. Then $\psi_{\widetilde{L},\,\alpha}(\lambda)$ in \eqref{eq:simple} along $y=0$ computes to
$$
\left.\psi_{\widetilde{L},\,\alpha}(\lambda)\right|_{y=0} =
    \begin{pmatrix}
    \frac{\lambda  \alpha ^2-2 \alpha +\lambda +\left(\alpha ^2-1\right) \lambda  \cos (2K\pi  x)}{2 (\alpha  \lambda -1) \left(\alpha \cos^2(K\pi x) + \sin^2(K\pi x)\right)} &
    \frac{\mi(\alpha -1)\sqrt{\alpha}(\lambda +1)\sin(2K\pi x)}
    {2(\alpha\lambda -1)
    \left(\alpha\cos^2(K\pi x) + \sin^2(K\pi x)\right)} \\
    -\frac{\mi(\alpha -1)\sqrt{\alpha}(\lambda +1)
    \sin(2K\pi x)}{2(\alpha \lambda -1)
    \left(\alpha\cos^2(K\pi x) + \sin^2(K\pi x)\right)} & -\frac{\alpha^2-2\lambda \alpha + \left(\alpha^2-1\right)
    \cos(2K\pi x)+1}{2(\alpha\lambda -1)\left(\alpha
    \cos^2(K\pi x) + \sin^2(K\pi x)\right)}
\end{pmatrix}
$$
Evaluating at $\lambda = 0$ gives
\begin{equation*}
    \left.\psi_{\widetilde{L},\,\alpha}(0)\right|_{y=0} =
    \begin{pmatrix}
    \frac{\alpha }{\alpha\cos^2(K\pi x) + \sin^2(K\pi x)} &
 \frac{\mi (1-\alpha)\sqrt{\alpha } \sin (2K\pi x)}{2\left(\alpha \cos^2(K\pi x)+\sin^2(K\pi x)\right)} \\
 \frac{\mi(\alpha -1)\sqrt{\alpha } \sin (2K\pi x)}{2 \left(\alpha
   \cos^2(K\pi x)+\sin^2(K\pi x)\right)} & \frac{\alpha^2 + \left(\alpha^2 - 1\right) \cos (2K\pi x)+1}{2\left(\alpha \cos^2(K\pi x) + \sin^2(K\pi x)\right)}
\end{pmatrix} \spaceperiod
\end{equation*}
Now $\det\left.\psi_{\widetilde{L},\,\alpha}(0)\right|_{y=0} = \alpha \neq 0$, and Gram-Schmidt on $\alpha^{-1/2}\psi_{\widetilde{L},\,\alpha}(0)|_{y=0} = Q\,R$ gives
\begin{equation*}
    Q = \begin{pmatrix}
        \frac{2 \sqrt{\alpha }}{\sqrt{(\alpha +1)^2-(\alpha -1)^2 \cos^2(2K\pi x)}} & \frac{\mi (\alpha -1) \sin (2K\pi x)}{\sqrt{(\alpha -1)^2 \sin^2(2K\pi x)+4 \alpha }} \\ \frac{\mi(\alpha -1) \sin (2\pi K x)}{\sqrt{(\alpha +1)^2-(\alpha -1)^2 \cos^2(2K\pi x)}} & \frac{2 \sqrt{\alpha }}{\sqrt{(\alpha
   -1)^2 \sin ^2(2K\pi x)+4 \alpha }}
   \end{pmatrix}\,.
\end{equation*}
Putting everything together gives that $h^{-1}_{\widetilde{L},\alpha}|_{y=0} = \sqrt{\det\psi_{\widetilde{L},\alpha}} \psi_{\widetilde{L},\alpha}^{-1} Q$ is equal to
\begin{equation*}
    \begin{pmatrix}
 \frac{\sqrt{2}\sqrt{\alpha}((\alpha +1)(\lambda -1)-(\alpha -1)
   (\lambda +1) \cos(2K\pi x))}{\sqrt{\lambda -\alpha }
   \sqrt{\alpha\lambda -1} \sqrt{-\cos (4K\pi x) (\alpha -1)^2 + \alpha(\alpha +6)+1}} & -\frac{\mi \left(\alpha ^2-1\right)
   \sin(2K\pi x)}{\sqrt{\lambda - \alpha }\sqrt{\alpha\lambda -1}
   \sqrt{(\alpha -1)^2 \sin ^2(2K\pi x) + 4 \alpha }} \\
 \frac{\mi\sqrt{2} \left(\alpha^2-1\right) \lambda  \sin(2K\pi x)}{\sqrt{\lambda -\alpha }\sqrt{\alpha\lambda -1} \sqrt{-\cos (4 K\pi x)(\alpha -1)^2+\alpha (\alpha +6)+1}} & \frac{\sqrt{2}
   \sqrt{\alpha}((\alpha +1)(\lambda -1)+(\alpha -1)(\lambda +1)
   \cos(2K\pi x))}{\sqrt{\lambda -\alpha }\sqrt{\alpha\lambda
   -1} \sqrt{-\cos (4K\pi x)(\alpha -1)^2+\alpha (\alpha +6)+1}}
    \end{pmatrix} \,.
\end{equation*}
Differentiating with respect to $\lambda$ gives that $h^{-1\,\prime}_{\widetilde{L},\alpha}|_{y=0,\lambda=1}$ is equal to
\begin{equation*}
    \begin{pmatrix}
     \frac{\sqrt{2}\sqrt{\alpha}(\alpha +1)}{\sqrt{-(\alpha-1)^2}
    \sqrt{-\cos(4K\pi x)(\alpha -1)^2 + \alpha(\alpha +6)+1}} &
    \frac{\mi\sqrt{\alpha -1}(\alpha +1)\sin(2K\pi x)}
    {2\sqrt{1-\alpha}\sqrt{(\alpha -1)^2 \sin^2(2K\pi x)+4\alpha}}
    \\
    \frac{\mi\sqrt{\alpha -1}(\alpha +1)\sin(2K\pi x)}
    {\sqrt{2-2\alpha}\sqrt{-\cos(4K\pi x)(\alpha -1)^2 +
    \alpha(\alpha +6)+1}} &
    \frac{\sqrt{2}\sqrt{\alpha}(\alpha +1)}
    {\sqrt{-(\alpha -1)^2}\sqrt{-\cos(4K\pi x)(\alpha -1)^2 +
    \alpha(\alpha +6)+1}}
\end{pmatrix}\,,
\end{equation*}
and consequently
\begin{equation*}
\left.h^{-1\,\prime}_{\widetilde{L},\alpha}\,  h_{\widetilde{L},\alpha} \right|_{y=0,\lambda=1} =
\begin{pmatrix}
    -\frac{\alpha ^2+(\alpha +1)^2 \cos(2K\pi x)-1}{2 \left(\cos (2K\pi x) (\alpha -1)^2+\alpha^2 - 1 \right)} & -\frac{\mi \sqrt{\alpha }(\alpha +1) \sin (2 K \pi  x)}{\cos (2 K \pi  x) (\alpha -1)^2 + \alpha^2-1} \\
    \frac{\mi\sqrt{\alpha}(\alpha +1) \sin(2K\pi x)}{\cos(2K\pi x)
   (\alpha -1)^2+\alpha^2-1} & \frac{\alpha^2+(\alpha +1)^2 \cos (2K\pi x)-1}{2 \left(\cos (2K\pi x) (\alpha -1)^2+\alpha^2 - 1 \right)}
\end{pmatrix} \,.
\end{equation*}
Thus the final contribution to the curve $y=0$ comes from
\begin{equation} \label{eq:part3}
    \left. \mi h_{L,\alpha} F_\lambda h_{\widetilde{L},\alpha}^{\prime -1} h_{\widetilde{L},\alpha} F_\lambda ^{-1} h_{L,\alpha}^{-1}\right|_{y=0,\lambda=1} =
    \begin{pmatrix} \mi u & - v \\ v & -\mi u \end{pmatrix}
\end{equation}
with the real valued functions $u$ and $v$ given by
\begin{equation*} \begin{split}
    u &= \tfrac{(\alpha +1) ((\alpha -1) \cos (2 \pi  K x)-\alpha -1) \left(\left(\sqrt{\alpha }-1\right)^2 \cos (2 \pi  (K+1)
   x)+\left(\sqrt{\alpha }+1\right)^2 \cos (2 \pi  (K-1) x)+2 (\alpha -1) \cos (2 \pi  x)\right)}{2 (\alpha -1) \left((\alpha
   -1)^2 (-\cos (4 \pi  K x))+\alpha ^2+6 \alpha +1\right)}\,, \\
   v &= \tfrac{(\alpha +1) \left(\left(\sqrt{\alpha }-1\right)^2 \sin (2 \pi  (K+1) x)-\left(\sqrt{\alpha }+1\right)^2 \sin (2 \pi
   (K-1) x)+2 (\alpha -1) \sin (2 \pi  x)\right) ((\alpha -1) \cos (2 \pi  K x)-\alpha -1)}{2 (\alpha -1) \left((\alpha -1)^2
   (-\cos (4 \pi  K x))+\alpha ^2+6 \alpha +1\right)}\,.
   \end{split}
\end{equation*}
Inspection of the three summands in \eqref{eq:3parts} computed in \eqref{eq:part1}, \eqref{eq:part2} and \eqref{eq:part3} show that the $y=0$ curve on the bubbleton is a planar curve, since the  off-diagonal terms do not have an imaginary part. Combining these terms then gives the planar curve $x \mapsto (X(x),\,Z(x))$ with
\begin{equation*} \begin{split}
    &X(x) = \tfrac{\sin(2\pi x)}{4}-\tfrac{(\alpha +1) \left(\left(\sqrt{\alpha }-1\right)^2 \sin (2\pi (K+1)x)-\left(\sqrt{\alpha }+1\right)^2 \sin (2 \pi  (K-1) x)+2 (\alpha -1) \sin (2 \pi  x)\right) ((\alpha -1) \cos (2\pi K x)-\alpha -1)}{2(\alpha -1)^3 \cos (4\pi Kx)-(\alpha -1) (\alpha (\alpha +6)+1)}\,,  \\
     &Z(x) = \tfrac{((\alpha^2 -1) \cos (2 \pi  K x)-(\alpha +1)^2) \left(\left(\sqrt{\alpha}-1\right)^2\cos(2\pi(K+1)x) + \left(\sqrt{\alpha }+1\right)^2 \cos(2\pi (K-1)x)+2(\alpha -1)\cos(2\pi x)\right)}{2(\alpha -1)\left((\alpha -1)^2(-\cos(4\pi K x)) +\alpha (\alpha + 6 ) + 1\right)} \\ &\qquad \qquad \qquad \hspace{3cm}-\tfrac{\sin^2(\pi x)}{2} - \tfrac{K}{2\sqrt{K^2-1}}\,.   
\end{split}
\end{equation*}
The turning number of this immersed planar curve $[0,\,1] \to \bbR^2,\,x \mapsto (X(x),\,Z(x))$ computes to
\begin{equation*}
    \frac{1}{2\pi} \int_0^1 \frac{X'(x)Z''(x) - X''(x)Z'(x)}{X'^2(x) + Z'^2(x)}\,dx\, = 2K - 1\,.
\end{equation*}
Since $K\geq 2$, the planar curve has self intersections, and thus the bubbleton is not embedded. Plots of the curves for $K=2,\,3,\,4,\,5$ are shown in figures \ref{fig:curves} and \ref{fig:curves2}.
\end{proof}
%
%
\begin{figure}[h]
\centering
\includegraphics[width=5.3cm]{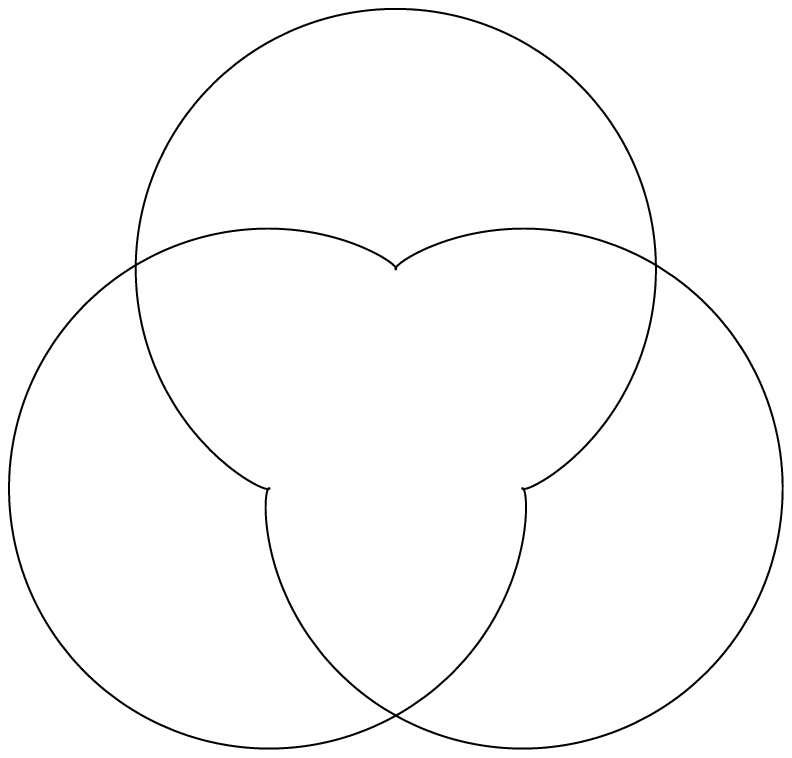}\hspace{0.0cm}
\includegraphics[width=5.3cm]{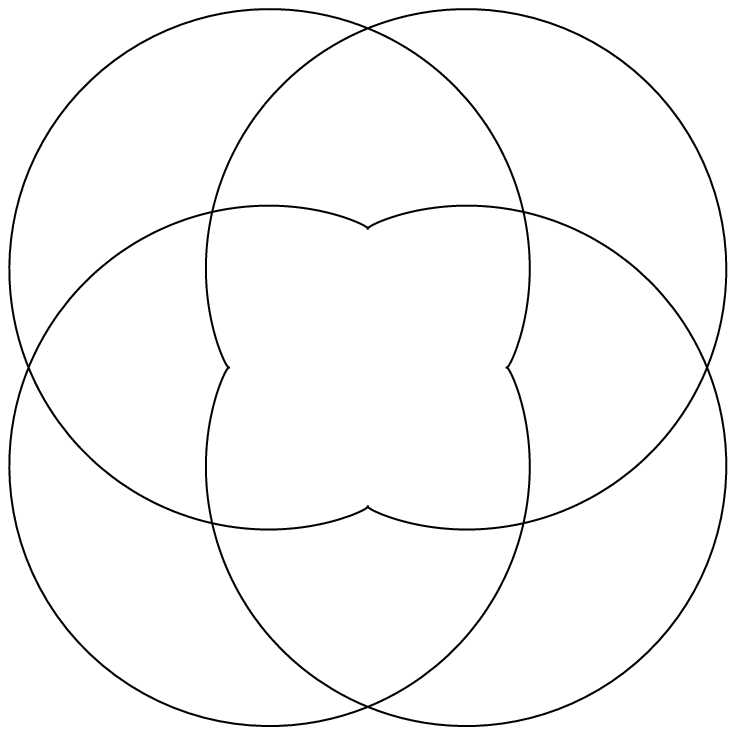}\hspace{0.0cm}
\includegraphics[width=5.3cm]{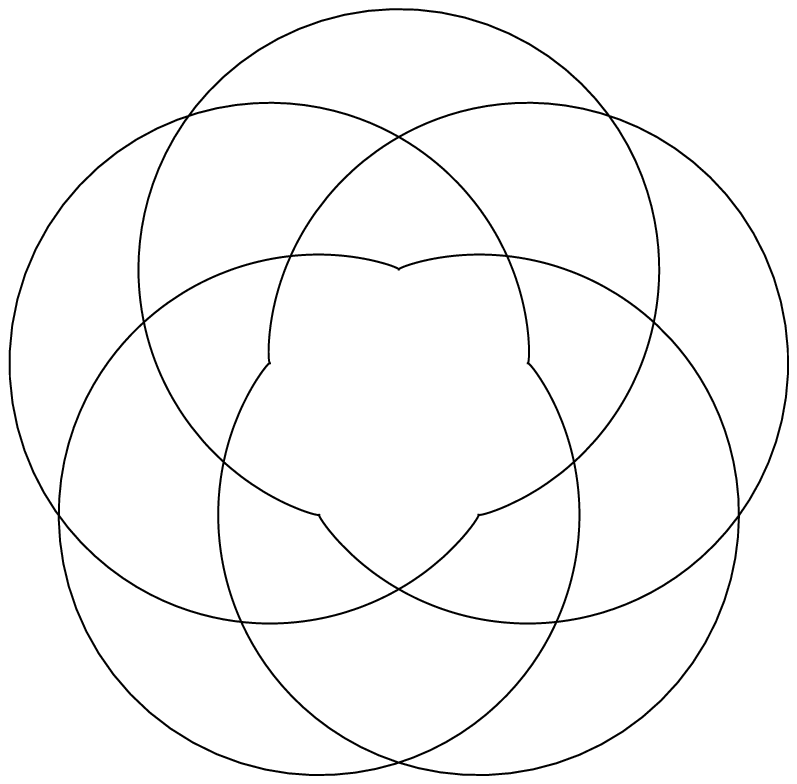}\hspace{0.0cm}
\caption{ \label{fig:curves2} Planar curves on the 3, 4 and 5-lobed single bubbletons, with turning numbers $5,\,7,\,9$ respectively. What appear to be cusps on the immersed curves are in fact small loops, as in the curve in figure \ref{fig:curves}.}
\end{figure}


\def\cydot{\leavevmode\raise.4ex\hbox{.}} \def\cprime{$'$}
\providecommand{\bysame}{\leavevmode\hbox to3em{\hrulefill}\thinspace}
\providecommand{\MR}{\relax\ifhmode\unskip\space\fi MR }
\providecommand{\MRhref}[2]{%
  \href{http://www.ams.org/mathscinet-getitem?mr=#1}{#2}
}
\providecommand{\href}[2]{#2}

\end{document}